\documentclass[oneside,english]{amsart}
\usepackage[T1]{fontenc}
\usepackage[latin9]{inputenc}
\usepackage{babel}
\usepackage{textcomp}
\usepackage{amstext}
\usepackage{amsthm}
\usepackage[unicode=true,pdfusetitle,
 bookmarks=true,bookmarksnumbered=false,bookmarksopen=false,
 breaklinks=false,pdfborder={0 0 1},backref=false,colorlinks=false]
 {hyperref}

\makeatletter
\numberwithin{equation}{section}
\theoremstyle{plain}
\newtheorem{thm}{\protect\theoremname}[section]

\makeatother

\providecommand{\theoremname}{Theorem}

\begin{document}
\title[Some $\Pi_{q}$-identities of Gosper]{Proofs for certain $\Pi_{q}$-conjectures of Gosper }
\author{Bing He}
\address{School of Mathematics and Statistics, Central South University \\
Changsha 410083, Hunan, People's Republic of China}
\email{yuhe001@foxmail.com; yuhelingyun@foxmail.com}
\begin{abstract}
In 2001 W. Gosper introduced a constant $\Pi_{q}$ and conjectured
without proofs many intriguing identities on this constant. In this
paper we establish some modular equations of degrees 3 and 5. From
these modular equations we confirm two groups of $\Pi_{q}$-identities
in Gosper's list. One group involves $\Pi_{q},\Pi_{q^{2}},\Pi_{q^{3}}$
or $\Pi_{q^{6}}$ while the other is related to $\Pi_{q},\Pi_{q^{2}},\Pi_{q^{5}}$
or $\Pi_{q^{10}}$. 
\end{abstract}

\keywords{$\Pi_{q}$-identity; modular equation}
\subjclass[2000]{33D15, 11F03, 14H42}
\maketitle

\section{Introduction}

Throughout this paper we assume that $|q|<1.$ W. Gosper \cite[p. 85]{G}
first introduced the $q$-constant $\Pi_{q}:$
\begin{equation}
\Pi_{q}=q^{1/4}\frac{(q^{2};q^{2})_{\infty}^{2}}{(q;q^{2})_{\infty}^{2}},\label{eq:0}
\end{equation}
where $(a;q)_{\infty}$ is defined by
\[
(a;q)_{\infty}=\prod_{n=0}^{\infty}(1-aq^{n}),
\]
and then stated without proofs many identities involving $\Pi_{q}$
\cite[pp. 102--104]{G} by employing empirical evidence based on a
computer program called MACSYMA. In particular, he \cite[pp. 103--104]{G}
conjectured the following interesting $\Pi_{q}$-identities:
\begin{align}
\frac{\Pi_{q}^{2}}{\Pi_{q^{2}}\Pi_{q^{4}}}-\frac{\Pi_{q^{2}}^{2}}{\Pi_{q^{4}}^{2}} & =4,\label{eq:1-1}\\
\Pi_{q^{3}}^{2}+3\Pi_{q}\Pi_{q^{9}} & =\sqrt{\Pi_{q}\Pi_{q^{9}}}(\Pi_{q}+3\Pi_{q^{9}}),\label{eq:1-6}\\
\frac{\Pi_{q^{2}}\Pi_{q^{3}}^{2}}{\Pi_{q^{6}}\Pi_{q}^{2}} & =\frac{\Pi_{q^{2}}-\Pi_{q^{6}}}{\Pi_{q^{2}}+3\Pi_{q^{6}}},\label{eq:11-2}\\
\Pi_{q^{2}}\Pi_{q^{3}}^{4} & =\Pi_{q^{6}}(\Pi_{q^{2}}-\Pi_{q^{6}})^{3}(\Pi_{q^{2}}+3\Pi_{q^{6}}),\label{eq:11-5}\\
\Pi_{q^{6}}\Pi_{q}^{4} & =\Pi_{q^{2}}(\Pi_{q^{2}}-\Pi_{q^{6}})(\Pi_{q^{2}}+3\Pi_{q^{6}})^{3}.\label{eq:11-6}
\end{align}

The formula \eqref{eq:1-1} was deduced by Gosper \cite[p. 93]{G}.
The $\Pi_{q}$-identity \eqref{eq:1-6} was confirmed by the author
and H.-C. Zhai \cite{HZ} by establishing an identity involving $q$-trigonometric
functions and $\Pi_{q},$ which is equivalent to a theta function
identity that can be proved by using an addition formula for Jacobi's
theta functions of Liu \cite[Theorem 1]{L2} (See \cite{G} for the
definitions of the $q$-trigonometric functions and see \cite{L1}
for applications of Liu's addition formula and the definitions of
Jacobi's theta functions). M. El Bachraoui \cite[Theorem 2.2]{E}
just gave a partial proof of the identity \eqref{eq:11-2}. Namely,
he proved that 
\[
\bigg(\frac{\Pi_{q^{2}}\Pi_{q^{3}}^{2}}{\Pi_{q^{6}}\Pi_{q}^{2}}\bigg)^{2}=\bigg(\frac{\Pi_{q^{2}}-\Pi_{q^{6}}}{\Pi_{q^{2}}+3\Pi_{q^{6}}}\bigg)^{2}.
\]
In \cite[Theorem 2.3]{E} El Bachraoui only showed that \eqref{eq:11-5}
is equivalent to \eqref{eq:11-6}. See \cite{AAE} and \cite{E} for
many other $\Pi_{q}$-identities not mentioned by Gosper \cite{G}.

In this paper we will consider the $\Pi_{q}$-identities \eqref{eq:11-2},
\eqref{eq:11-5}, \eqref{eq:11-6} and many other $\Pi_{q}$-identities
of Gosper and adopt the notations of \cite[Chapters 5 and 6]{B}.

The definition of modular equations \cite[(6.3.2)]{B} is very important.

Let $0<k,\:l<1$ and let $n$ be a positive integer. A relation between
$k$ and $l$ induced by the formula
\[
n\frac{_{2}F_{1}(1/2,1/2;1;1-k^{2})}{_{2}F_{1}(1/2,1/2;1;k^{2})}=\frac{_{2}F_{1}(1/2,1/2;1;1-l^{2})}{_{2}F_{1}(1/2,1/2;1;l^{2})}
\]
is called a modular equation of degree $n.$ Take $\alpha=k^{2},\:\beta=l^{2},$
we say that $\beta$ has degree $n$ over $\alpha.$ The multiplier
$m$ is given by
\[
m=\frac{z_{1}}{z_{n}},
\]
where 
\[
z_{n}=\varphi^{2}(q^{n})
\]
and 
\[
\varphi(q)=\sum_{k=-\infty}^{\infty}q^{k^{2}}.
\]

In Sections \ref{sec:tw} and \ref{sec:th} we will consider many
$\Pi_{q}$-identities conjectured by W. Gosper. These identities can
be divided into two groups, one group of identities involving the
$q$-constants $\Pi_{q},\Pi_{q^{2}},\Pi_{q^{3}}$ or $\Pi_{q^{6}}$
and the other concerning $\Pi_{q},\Pi_{q^{2}},\Pi_{q^{5}}$ or $\Pi_{q^{10}}$.
They are confirmed by establishing several modular equations of degrees
3 or 5.

\section{\label{sec:tw} Identities involving $\Pi_{q},\Pi_{q^{2}},\Pi_{q^{3}}$
or $\Pi_{q^{6}}$}

\subsection{Statement of results}

In \cite[p. 103]{G} W. Gosper conjectured the following interesting
identities:

\begin{align}
\sqrt{\Pi_{q^{2}}\Pi_{q^{6}}}(\Pi_{q}^{2}-3\Pi_{q^{3}}^{2}) & =\sqrt{\Pi_{q}\Pi_{q^{3}}}(\Pi_{q^{2}}^{2}+3\Pi_{q^{6}}^{2}),\label{eq:11-4}\\
\Pi_{q^{2}}^{2}(\Pi_{q}^{4}+18\Pi_{q}^{2}\Pi_{q^{3}}^{2}-27\Pi_{q^{3}}^{4}) & =\Pi_{q}\Pi_{q^{3}}(\Pi_{q}^{4}+16\Pi_{q^{2}}^{4}),\label{eq:11-7}\\
\Pi_{q^{6}}^{2}(\Pi_{q}^{4}-6\Pi_{q}^{2}\Pi_{q^{3}}^{2}-3\Pi_{q^{3}}^{4}) & =\Pi_{q}\Pi_{q^{3}}(\Pi_{q^{3}}^{4}+16\Pi_{q^{6}}^{4}),\label{eq:11-8}\\
\Pi_{q}\Pi_{q^{3}}(\Pi_{q}^{2}\pm4\Pi_{q^{2}}^{2})^{2} & =\Pi_{q^{2}}^{2}(\Pi_{q}\mp\Pi_{q^{3}})(\Pi_{q}\pm3\Pi_{q^{3}})^{3},\label{eq:11-9}\\
\Pi_{q}\Pi_{q^{3}}(\Pi_{q^{3}}^{2}\pm4\Pi_{q^{6}}^{2})^{2} & =\Pi_{q^{6}}^{2}(\Pi_{q}\mp\Pi_{q^{3}})^{3}(\Pi_{q}\pm3\Pi_{q^{3}}).\label{eq:11-10}
\end{align}
In this section we will confirm these identities by establishing the
following theorem.
\begin{thm}
\label{t2-1} The identities \eqref{eq:11-2}\textendash \eqref{eq:11-6}
and \eqref{eq:11-4}\textendash \eqref{eq:11-10} are true.
\end{thm}
Both of the identities \eqref{eq:11-2} and \eqref{eq:11-4} involve
all of the four constants $\Pi_{q},\Pi_{q^{2}},\Pi_{q^{3}}$ and $\Pi_{q^{6}},$
however, each of the other identities in Theorem \ref{t2-1} contains
only three of these four constants. We will show these $\Pi_{q}$-identities
by establishing several modular equations of degree 3.

\subsection{Auxiliary results}

Some auxiliary results are required to prove Theorem \ref{t2-1}.
\begin{thm}
\label{t2-2} Let $\beta$ have degree 3 over $\alpha.$ Then
\begin{align}
m-\sqrt{\frac{\beta}{\alpha}} & =1+\frac{3}{m}\sqrt{\frac{\beta}{\alpha}},\label{eq:4-1}\\
m^{2}\alpha+3\beta & =2(\alpha\beta)^{1/8}(m^{2}\alpha^{1/2}-3\beta{}^{1/2}),\label{eq:4-2}\\
\frac{16}{m}\sqrt{\frac{\beta}{\alpha}} & =\alpha\bigg(1-\frac{1}{m}\sqrt{\frac{\beta}{\alpha}}\bigg)\bigg(1+\frac{3}{m}\sqrt{\frac{\beta}{\alpha}}\bigg)^{3},\label{eq:4-3}\\
\frac{1}{m}\sqrt[4]{\frac{\beta}{\alpha}}\bigg(1+\alpha\bigg) & =\frac{\alpha{}^{1/2}}{4}\bigg(1+\frac{18}{m^{2}}\sqrt{\frac{\beta}{\alpha}}-\frac{27}{m^{4}}\frac{\beta}{\alpha}\bigg),\label{eq:4-4}\\
m\sqrt[4]{\frac{\alpha}{\beta}}\bigg(1+\beta\bigg) & =\frac{\beta{}^{1/2}}{4}\bigg(\frac{m^{4}\alpha}{\beta}-6m^{2}\sqrt{\frac{\alpha}{\beta}}-3\bigg),\label{eq:4-5}\\
\frac{1}{m}\sqrt[4]{\frac{\beta}{\alpha}}\bigg(1\pm\alpha{}^{1/2}\bigg)^{2} & =\frac{\alpha{}^{1/2}}{4}\bigg(1\mp\frac{1}{m}\sqrt[4]{\frac{\beta}{\alpha}}\bigg)\bigg(1\pm\frac{3}{m}\sqrt[4]{\frac{\beta}{\alpha}}\bigg)^{3},\label{eq:4-6}\\
m\sqrt[4]{\frac{\alpha}{\beta}}(1\pm\beta{}^{1/2})^{2} & =\frac{\beta{}^{1/2}}{4}\bigg(m\sqrt[4]{\frac{\alpha}{\beta}}\mp1\bigg)^{3}\bigg(m\sqrt[4]{\frac{\alpha}{\beta}}\pm3\bigg).\label{eq:44-7}
\end{align}
\end{thm}
\noindent{\it Proof.} The identity \eqref{eq:4-1} follows easily
from \cite[(6.3.23)]{B}.

We now show \eqref{eq:4-2}. It follows from \cite[(6.3.19) and (6.3.20)]{B}
that 
\begin{align*}
 & (\alpha\beta)^{1/8}(m^{2}\alpha^{1/2}-3\beta{}^{1/2})\\
 & =\frac{(m-1)(3+m)}{8m}\bigg((3+m)m-3(m-1)\bigg)\\
 & =\frac{(m-1)(3+m)(m^{2}+3)}{8m}
\end{align*}
and 
\begin{align*}
 & m^{2}\alpha+3\beta\\
 & =\frac{(m-1)(3+m)}{16m}((3+m)^{2}+3(m-1)^{2})\\
 & =\frac{(m-1)(3+m)(m^{2}+3)}{4m}.
\end{align*}
From these identities the formula \eqref{eq:4-2} follows readily.

The identity \eqref{eq:4-3} can be obtained by combining \cite[(6.3.19)]{B}
and \cite[(6.3.23)]{B}.

We now prove \eqref{eq:4-4}. It can be deduced from \cite[(6.3.19) and (6.3.23)]{B}
that
\begin{align*}
\frac{1}{m}\sqrt[4]{\frac{\beta}{\alpha}}\bigg(1+\alpha\bigg) & =\frac{1}{m}\sqrt{\frac{m(m-1)}{3+m}}\bigg(1+\frac{(m-1)(3+m)^{3}}{16m^{3}}\bigg)\\
 & =\frac{m^{4}+24m^{3}+18m^{2}-27}{16m^{3}}\sqrt{\frac{m-1}{m(3+m)}}
\end{align*}
and 
\begin{align*}
 & \frac{\alpha{}^{1/2}}{4}\bigg(1+\frac{18}{m^{2}}\sqrt{\frac{\beta}{\alpha}}-\frac{27}{m^{4}}\frac{\beta}{\alpha}\bigg)\\
 & =\frac{3+m}{16m}\sqrt{\frac{(m-1)(3+m)}{m}}\bigg(1+\frac{18(m-1)}{m(3+m)}-\frac{27}{m^{2}}\frac{(m-1)^{2}}{(3+m)^{2}}\bigg)\\
 & =\frac{m^{4}+24m^{3}+18m^{2}-27}{16m^{3}}\sqrt{\frac{m-1}{m(3+m)}}.
\end{align*}
From these two identities \eqref{eq:4-4} follows quickly.

We then deduce \eqref{eq:4-5}. It follows from \cite[(6.3.20) and (6.3.23)]{B}
that 
\begin{align*}
m\sqrt[4]{\frac{\alpha}{\beta}}\bigg(1+\beta\bigg) & =\bigg(\frac{m(3+m)}{m-1}\bigg)^{1/2}\bigg(1+\frac{(m-1)^{3}(3+m)}{16m}\bigg)\\
 & =\frac{m^{4}-6m^{2}+24m-3}{16}\bigg(\frac{3+m}{m(m-1)}\bigg)^{1/2}
\end{align*}
and
\begin{align*}
 & \frac{\beta{}^{1/2}}{4}\bigg(\frac{m^{4}\alpha}{\beta}-6m^{2}\sqrt{\frac{\alpha}{\beta}}-3\bigg)\\
 & =\frac{m-1}{16}\bigg(\frac{(m-1)(3+m)}{m}\bigg)^{1/2}\bigg(\frac{m^{2}(3+m)^{2}}{(m-1)^{2}}-\frac{6m(3+m)}{m-1}-3\bigg)\\
 & =\frac{m^{4}-6m^{2}+24m-3}{16}\bigg(\frac{3+m}{m(m-1)}\bigg)^{1/2},
\end{align*}
from which \eqref{eq:4-5} is obtained.

We now derive \eqref{eq:4-6}. Using \cite[(6.3.19) and (6.3.23)]{B}
we get
\begin{align*}
 & \frac{1}{m}\sqrt[4]{\frac{\beta}{\alpha}}\bigg(1\pm\alpha{}^{1/2}\bigg)^{2}\\
 & =\sqrt{\frac{m-1}{m(3+m)}}\bigg(1\pm\frac{3+m}{4m}\sqrt{\frac{(m-1)(3+m)}{m}}\bigg)^{2}\\
 & =\frac{m^{4}+24m^{3}+18m^{2}-27}{16m^{3}}\sqrt{\frac{m-1}{m(3+m)}}\pm\frac{m^{2}+2m-3}{2m^{2}}
\end{align*}
and
\begin{align*}
 & \frac{\alpha{}^{1/2}}{4}\bigg(1\mp\frac{1}{m}\sqrt[4]{\frac{\beta}{\alpha}}\bigg)\bigg(1\pm\frac{3}{m}\sqrt[4]{\frac{\beta}{\alpha}}\bigg)^{3}\\
 & =\frac{3+m}{16m}\sqrt{\frac{(m-1)(3+m)}{m}}\bigg(1\mp\sqrt{\frac{m-1}{m(3+m)}}\bigg)\bigg(1\pm3\sqrt{\frac{m-1}{m(3+m)}}\bigg)^{3}\\
 & =\frac{m^{4}+24m^{3}+18m^{2}-27}{16m^{3}}\sqrt{\frac{m-1}{m(3+m)}}\pm\frac{m^{2}+2m-3}{2m^{2}}.
\end{align*}
From these two identities we obtain \eqref{eq:4-6}.

Finally, we show \eqref{eq:44-7}. We deduce from \cite[(6.3.20) and (6.3.23)]{B}
that
\begin{align*}
m\sqrt[4]{\frac{\alpha}{\beta}}(1\pm\beta{}^{1/2})^{2} & =\sqrt{\frac{m(3+m)}{m-1}}\bigg(1\pm\frac{m-1}{4}\sqrt{\frac{(m-1)(3+m)}{m}}\bigg)^{2}\\
 & =\frac{m^{4}-6m^{2}+24m-3}{16m}\sqrt{\frac{m(3+m)}{m-1}}\pm\frac{m^{2}+2m-3}{2}
\end{align*}
and
\begin{align*}
 & \frac{\beta{}^{1/2}}{4}\bigg(m\sqrt[4]{\frac{\alpha}{\beta}}\mp1\bigg)^{3}\bigg(m\sqrt[4]{\frac{\alpha}{\beta}}\pm3\bigg)\\
 & =\frac{m-1}{16}\sqrt{\frac{(m-1)(3+m)}{m}}\bigg(\sqrt{\frac{m(3+m)}{m-1}}\mp1\bigg)^{3}\bigg(\sqrt{\frac{m(3+m)}{m-1}}\pm3\bigg)\\
 & =\frac{m^{4}-6m^{2}+24m-3}{16m}\sqrt{\frac{m(3+m)}{m-1}}\pm\frac{m^{2}+2m-3}{2}.
\end{align*}
From these we arrive at \eqref{eq:44-7}. This completes the proof
of Theorem \ref{t2-2}. \qed

\subsection{Proof of Theorem \ref{t2-1}}

In this subsection, we only consider modular equations of degree 3
and then we always assumed that $n=3$ and
\begin{equation}
m=\frac{z_{1}}{z_{3}}.\label{eq:t22-0}
\end{equation}

We are now ready to show Theorem \ref{t2-1}.

\noindent{\it Proof of Theorem \ref{t2-1}.} According to \cite[Theorem 2.3]{E},
we know that \eqref{eq:11-5} is equivalent to \eqref{eq:11-6}, so
we only need to prove one of them. In this section we shall show the
identities \eqref{eq:11-2}, \eqref{eq:11-6}, \eqref{eq:11-4}\textendash \eqref{eq:11-10}.
If the identities in Theorem \ref{t2-1} hold for $0<q<1,$ then,
by analytic continuation, these identities are also true for $|q|<1.$
So we can assume that $0<q<1.$

Let 
\[
\psi(q)=\sum_{n=0}^{\infty}q^{n(n+1)/2}.
\]
It follows from \eqref{eq:0} and \cite[(1.3.14)]{B} that 
\begin{equation}
\Pi_{q}=q^{1/4}\psi^{2}(q).\label{eq:t0-1}
\end{equation}
Then the identities \eqref{eq:11-2}, \eqref{eq:11-6}, \eqref{eq:11-4}\textendash \eqref{eq:11-10}
are respectively equivalent to 
\begin{align}
\frac{\psi^{2}(q^{2})\psi^{4}(q^{3})}{\psi^{2}(q^{6})\psi^{4}(q)} & =\frac{\psi^{2}(q^{2})-q\psi^{2}(q^{6})}{\psi^{2}(q^{2})+3q\psi^{2}(q^{6})},\label{eq:21-1}\\
\frac{\psi^{2}(q^{6})}{\psi^{2}(q^{2})}\frac{\psi^{8}(q)}{\psi^{8}(q^{2})} & =\bigg(1-q\frac{\psi^{2}(q^{6})}{\psi^{2}(q^{2})}\bigg)\bigg(1+3q\frac{\psi^{2}(q^{6})}{\psi^{2}(q^{2})}\bigg)^{3},\label{eq:21-2}\\
\psi(q^{2})\psi(q^{6})(\psi^{4}(q)-3q\psi^{4}(q^{3})) & =\psi(q)\psi(q^{3})(\psi^{4}(q^{2})+3q^{2}\psi^{4}(q^{6})),\label{eq:21-3}\\
\frac{\psi^{2}(q^{3})}{\psi^{2}(q)}\bigg(1+16q\frac{\psi^{8}(q^{2})}{\psi^{8}(q)}\bigg) & =\frac{\psi^{4}(q^{2})}{\psi^{4}(q)}\bigg(1+18q\frac{\psi^{4}(q^{3})}{\psi^{4}(q)}-27q^{2}\frac{\psi^{8}(q^{3})}{\psi^{8}(q)}\bigg),\label{eq:21-4}\\
\frac{\psi^{2}(q)}{\psi^{2}(q^{3})}\bigg(1+16q^{3}\frac{\psi^{8}(q^{6})}{\psi^{8}(q^{3})}\bigg) & =\frac{\psi^{4}(q^{6})}{\psi^{4}(q^{3})}\bigg(\frac{\psi^{8}(q)}{\psi^{8}(q^{3})}-6q\frac{\psi^{4}(q)}{\psi^{4}(q^{3})}-3q^{2}\bigg),\label{eq:21-5}\\
\frac{\psi^{2}(q^{3})}{\psi^{2}(q)}\bigg(1\pm4q^{1/2}\frac{\psi^{4}(q^{2})}{\psi^{4}(q)}\bigg)^{2} & =\frac{\psi^{4}(q^{2})}{\psi^{2}(q)}\bigg(1\mp q^{1/2}\frac{\psi^{2}(q^{3})}{\psi^{2}(q)}\bigg)\bigg(1\pm3q^{1/2}\frac{\psi^{2}(q^{3})}{\psi^{2}(q)}\bigg)^{3},\label{eq:21-6}\\
\frac{\psi^{2}(q)}{\psi^{2}(q^{3})}\bigg(1\pm4q^{3/2}\frac{\psi^{4}(q^{6})}{\psi^{4}(q^{3})}\bigg)^{2} & =\frac{\psi^{4}(q^{6})}{\psi^{4}(q^{3})}\bigg(\frac{\psi^{2}(q)}{\psi^{2}(q^{3})}\mp q^{1/2}\bigg)^{3}\bigg(\frac{\psi^{2}(q)}{\psi^{2}(q^{3})}\pm3q^{1/2}\bigg).\label{eq:21-7}
\end{align}

We first prove \eqref{eq:21-1}. Let $\beta$ have degree 3 over $\alpha.$
Then, by \cite[Theorem 5.4.2, (i) and (iii)]{B},
\begin{align}
\psi(q) & =\sqrt{\frac{z_{1}}{2}}(\alpha/q)^{1/8},\label{eq:33-1}\\
\psi(q^{2}) & =\frac{1}{2}\sqrt{z_{1}}(\alpha/q)^{1/4},\label{eq:33-2}\\
\psi(q^{3}) & =\sqrt{\frac{z_{3}}{2}}(\beta/q^{3})^{1/8},\label{eq:33-3}\\
\psi(q^{6}) & =\frac{1}{2}\sqrt{z_{3}}(\beta/q^{3})^{1/4},\label{eq:33-4}
\end{align}
and so
\begin{align}
\frac{\psi(q^{3})}{\psi(q)} & =q^{-1/4}\sqrt{\frac{z_{3}}{z_{1}}}(\beta/\alpha)^{1/8},\label{eq:33-7}\\
\frac{\psi(q^{6})}{\psi(q^{2})} & =q^{-1/2}\sqrt{\frac{z_{3}}{z_{1}}}(\beta/\alpha)^{1/4}.\label{eq:33-5}
\end{align}
Hence, the formula \eqref{eq:21-1} follows by dividing both sides
of \eqref{eq:4-1} by $m(1+\frac{3}{m}\sqrt{\frac{\beta}{\alpha}})$
and then using \eqref{eq:33-7}, \eqref{eq:33-5} and \eqref{eq:t22-0}
in the resulting identity.

We now prove \eqref{eq:21-2}. It follows from \eqref{eq:33-1} and
\eqref{eq:33-2} that
\begin{equation}
\frac{\psi(q)}{\psi(q^{2})}=\frac{\sqrt{2}}{(\alpha/q)^{1/8}}.\label{eq:33-6}
\end{equation}
Then \eqref{eq:21-2} can be obtained by dividing both sides of \eqref{eq:4-3}
by $\alpha$ and then employing \eqref{eq:33-5}, \eqref{eq:33-6}
and \eqref{eq:t22-0} in the resulting identity.

The identity \eqref{eq:21-3} follows easily by multiplying both sides
of \eqref{eq:4-2} by $\frac{(\alpha\beta)^{1/8}\sqrt{z_{1}z_{3}/q^{3}}z_{3}^{2}}{32}$
and then using \eqref{eq:t22-0} in the resulting equation.

The formula \eqref{eq:21-4} can be deduced by dividing both sides
of \eqref{eq:4-4} by $\sqrt{q}$ and then using \eqref{eq:33-7}
and \eqref{eq:t22-0}.

We then show \eqref{eq:21-5}. It follows from \eqref{eq:33-6} that
\[
\frac{\psi(q^{3})}{\psi(q^{6})}=\frac{\sqrt{2}}{(\beta/q^{3})^{1/8}}
\]
and so
\begin{equation}
\frac{\psi(q^{6})}{\psi(q^{3})}=\frac{(\beta/q^{3})^{1/8}}{\sqrt{2}}.\label{eq:33-8}
\end{equation}
We multiply both sides of \eqref{eq:4-5} by $q^{1/2}$ and then apply
\eqref{eq:33-7}, \eqref{eq:33-8} and \eqref{eq:t22-0} in the resulting
identity to obtain \eqref{eq:21-5}.

The identity \eqref{eq:21-6} can be derived by dividing both sides
of \eqref{eq:4-6} by $q^{1/2}$ and then using \eqref{eq:33-7},
\eqref{eq:33-6} and \eqref{eq:t22-0} in the resulting formula.

The identity \eqref{eq:21-7} follows readily by multiplying both
sides of \eqref{eq:44-7} by $q^{1/2}$ and then employing \eqref{eq:33-7},
\eqref{eq:33-8} and \eqref{eq:t22-0} in the resulting identity.
This finishes the proof of Theorem \ref{t2-1}. \qed

\section{\label{sec:th} Identities involving $\Pi_{q},\Pi_{q^{2}},\Pi_{q^{5}}$
or $\Pi_{q^{10}}$}

\subsection{Statement of results}

Gosper \cite[pp. 103--104]{G} conjectured the following $\Pi_{q}$-identities:
\begin{align}
\Pi_{q^{2}}\Pi_{q^{5}}^{4}(16\Pi_{q^{10}}^{4}-\Pi_{q^{5}}^{4}) & =\Pi_{q^{10}}^{3}(5\Pi_{q^{10}}-\Pi_{q^{2}})(\Pi_{q^{2}}-\Pi_{q^{10}})^{5},\label{eq:1-7}\\
\Pi_{q^{10}}\Pi_{q}^{4}(16\Pi_{q^{2}}^{4}-\Pi_{q}^{4}) & =\Pi_{q^{2}}^{3}(5\Pi_{q^{10}}-\Pi_{q^{2}})^{5}(\Pi_{q^{2}}-\Pi_{q^{10}}),\label{eq:1-2}\\
\Pi_{q}\Pi_{q^{5}}(16\Pi_{q^{2}}^{4}-\Pi_{q}^{4})^{2} & =\Pi_{q^{2}}^{4}(5\Pi_{q^{5}}-\Pi_{q})^{5}(\Pi_{q^{5}}-\Pi_{q}),\label{eq:1-3}\\
\Pi_{q}\Pi_{q^{5}}(16\Pi_{q^{10}}^{4}-\Pi_{q^{5}}^{4})^{2} & =\Pi_{q^{10}}^{4}(5\Pi_{q^{5}}-\Pi_{q})(\Pi_{q^{5}}-\Pi_{q})^{5},\label{eq:1-4}\\
(\Pi_{q}\Pi_{q^{10}}-\Pi_{q^{2}}\Pi_{q^{5}})^{2} & =\Pi_{q^{2}}\Pi_{q^{10}}(\Pi_{q^{5}}-\Pi_{q})(5\Pi_{q^{5}}-\Pi_{q}).\label{eq:1-5}
\end{align}
In this section we will confirm these results.
\begin{thm}
\label{t1} The identities \eqref{eq:1-7}\textendash \eqref{eq:1-5}
are true.
\end{thm}
The identities \eqref{eq:1-7}\textendash \eqref{eq:1-4} only contain
three of the constants $\Pi_{q},\Pi_{q^{2}},\Pi_{q^{5}}$ and $\Pi_{q^{10}},$
but the formula \eqref{eq:1-5} includes all of these four constants.
These five identities have similar styles so that our proofs share
the same pattern. We will show these identities by setting up some
modular equations of degree 5.

\subsection{One lemma}

The value of the multiplier $m$ depends on $n$, but throughout this
subsection and the next subsection we only consider modular equations
of degree 5, then it is always assumed that $n=5$ and
\begin{equation}
m=\frac{z_{1}}{z_{5}}.\label{eq:t2-0}
\end{equation}
In order to prove Theorem \ref{t1} we need several auxiliary results.
\begin{thm}
\label{t2} If $\beta$ has degree $5$ over $\alpha,$ then
\begin{align}
256\frac{z_{1}}{z_{5}}\bigg(\frac{\alpha}{\beta}\bigg)^{1/2}\frac{1}{\beta}\bigg(1-\frac{1}{\beta}\bigg) & =\bigg(5-\frac{z_{1}}{z_{5}}\bigg(\frac{\alpha}{\beta}\bigg)^{1/2}\bigg)\bigg(\frac{z_{1}}{z_{5}}\bigg(\frac{\alpha}{\beta}\bigg)^{1/2}-1\bigg)^{5},\label{eq:2-1}\\
256\frac{z_{5}}{z_{1}}\bigg(\frac{\beta}{\alpha}\bigg)^{1/2}\frac{1}{\alpha}\bigg(1-\frac{1}{\alpha}\bigg) & =\bigg(5\frac{z_{5}}{z_{1}}\bigg(\frac{\beta}{\alpha}\bigg)^{1/2}-1\bigg)^{5}\bigg(1-\frac{z_{5}}{z_{1}}\bigg(\frac{\beta}{\alpha}\bigg)^{1/2}\bigg),\label{eq:2-2}\\
\frac{z_{5}}{z_{1}}\bigg(\frac{\beta}{\alpha}\bigg)^{1/4}\bigg(\alpha-1\bigg)^{2} & =\frac{\alpha}{16}\bigg(5\frac{z_{5}}{z_{1}}\bigg(\frac{\beta}{\alpha}\bigg)^{1/4}-1\bigg)^{5}\bigg(\frac{z_{5}}{z_{1}}\bigg(\frac{\beta}{\alpha}\bigg)^{1/4}-1\bigg),\label{eq:2-3}\\
\frac{z_{1}}{z_{5}}\bigg(\frac{\alpha}{\beta}\bigg)^{1/4}\bigg(\beta-1\bigg)^{2} & =\frac{\beta}{16}\bigg(5-\frac{z_{1}}{z_{5}}\bigg(\frac{\alpha}{\beta}\bigg)^{1/4}\bigg)\bigg(1-\frac{z_{1}}{z_{5}}\bigg(\frac{\alpha}{\beta}\bigg)^{1/4}\bigg)^{5},\label{eq:2-4}\\
\bigg(\frac{z_{1}}{z_{5}}\bigg(\frac{\alpha}{\beta}\bigg)^{1/4}-\frac{z_{1}}{z_{5}}\bigg(\frac{\alpha}{\beta}\bigg)^{1/2}\bigg)^{2} & =\frac{z_{1}}{z_{5}}\bigg(\frac{\alpha}{\beta}\bigg)^{1/2}\bigg(1-\frac{z_{1}}{z_{5}}\bigg(\frac{\alpha}{\beta}\bigg)^{1/4}\bigg)\bigg(5-\frac{z_{1}}{z_{5}}\bigg(\frac{\alpha}{\beta}\bigg)^{1/4}\bigg).\label{eq:2-5}
\end{align}
\end{thm}
\begin{proof}
We first prove \eqref{eq:2-1} and \eqref{eq:2-4}. According to \cite[Chapter 19, (13.12)--(13.15)]{B2}
we have
\begin{align}
\left(\frac{\alpha}{\beta}\right)^{1/4} & =\frac{2m+\rho}{m(m-1)},\label{eq:t2-1}\\
\left(\frac{\beta}{\alpha}\right)^{1/4} & =\frac{2m-\rho}{5-m},\label{eq:t2-4}\\
\left(\frac{1-\beta}{1-\alpha}\right)^{1/4} & =\frac{2m+\rho}{5-m},\nonumber \\
(\alpha\beta)^{1/2} & =\frac{4m^{3}-16m^{2}+20m+\rho(m^{2}-5)}{16m^{2}},\label{eq:t2-6}\\
\{(1-\alpha)(1-\beta)\}^{1/2} & =\frac{4m^{3}-16m^{2}+20m-\rho(m^{2}-5)}{16m^{2}},\label{eq:t2-5}
\end{align}
where 
\[
\rho=(m^{3}-2m^{2}+5m)^{1/2}.
\]
Then
\begin{align}
\beta & =\bigg(\frac{2m-\rho}{5-m}\bigg)^{2}\frac{4m^{3}-16m^{2}+20m+\rho(m^{2}-5)}{16m^{2}},\label{eq:t2-2}\\
1-\beta & =\bigg(\frac{2m+\rho}{5-m}\bigg)^{2}\frac{4m^{3}-16m^{2}+20m-\rho(m^{2}-5)}{16m^{2}}.\label{eq:t2-3}
\end{align}
Substituting \eqref{eq:t2-0}, \eqref{eq:t2-1}, \eqref{eq:t2-2}
and \eqref{eq:t2-3} into both sides of each of the identities \eqref{eq:2-1}
and \eqref{eq:2-4}, noticing that $\rho=(m^{3}-2m^{2}+5m)^{1/2}$
and then simplifying we find that both sides of each of the identities
\eqref{eq:2-1} and \eqref{eq:2-4} are respectively equal to 
\[
\bigg(\frac{2}{m-1}\bigg)^{2}A(m)
\]
 and 
\[
\frac{m-5}{256m(m-1)}B(m),
\]
where 
\begin{align*}
A(m) & =4m^{9}-24m^{8}+m^{8}\rho+64m^{7}-2m^{7}\rho-200m^{6}+6m^{6}\rho-40m^{5}-98m^{5}\rho\\
 & -40m^{4}+80m^{4}\rho-1280m^{3}-470m^{3}\rho-504m^{2}-470m^{2}\rho-28m-70m\rho-\rho
\end{align*}
and 
\begin{align*}
B(m) & =2m^{6}-10m^{5}+m^{5}\rho-5m^{4}\rho+4m^{4}+10m^{3}\rho\\
 & -4m^{3}-102m^{2}-42m^{2}\rho-18m-27m\rho-\rho.
\end{align*}
These prove \eqref{eq:2-1} and \eqref{eq:2-4}.

We now show \eqref{eq:2-2} and \eqref{eq:2-3}. According to \cite[Chapter 19, (13.12)]{B2}
we get
\begin{equation}
\left(\frac{1-\alpha}{1-\beta}\right)^{1/4}=\frac{2m-\rho}{m(m-1)}.\label{eq:t2-7}
\end{equation}
It follows from \eqref{eq:t2-1}, \eqref{eq:t2-6}, \eqref{eq:t2-5}
and \eqref{eq:t2-7} that 
\begin{align}
\alpha & =\bigg(\frac{2m+\rho}{m(m-1)}\bigg)^{2}\frac{4m^{3}-16m^{2}+20m+\rho(m^{2}-5)}{16m^{2}},\label{eq:t2-8}\\
1-\alpha & =\bigg(\frac{2m-\rho}{m(m-1)}\bigg)^{2}\frac{4m^{3}-16m^{2}+20m-\rho(m^{2}-5)}{16m^{2}}.\label{eq:t2-9}
\end{align}
We subsitute \eqref{eq:t2-0}, \eqref{eq:t2-4}, \eqref{eq:t2-8}
and \eqref{eq:t2-9} into both sides of each of the identities \eqref{eq:2-2}
and \eqref{eq:2-3}, note that $\rho=(m^{3}-2m^{2}+5m)^{1/2}$ and
then simplify to deduce that both sides of each of the identities
\eqref{eq:2-2} and \eqref{eq:2-3} equal
\[
-\frac{2^{12}m^{2}}{(m-5)^{12}}C(m)
\]
and 
\[
\frac{1-m}{256m^{6}(m-5)}D(m)
\]
respectively, where 
\begin{align*}
C(m) & =28m^{9}+2520m^{8}\text{\textminus}m^{8}\rho+32000m^{7}\text{\textminus}350m^{7}\rho+5000m^{6}\text{\textminus}11750m^{6}\rho\\
 & +25000m^{5}\text{\textminus}58750m^{5}\rho+625000m^{4}+50000m^{4}\rho\text{\textminus}1000000m^{3}\text{\textminus}306250m^{3}\rho\\
 & +1875000m^{2}+93750m^{2}\rho\text{\textminus}1562500m\text{\textminus}156250m\rho+390625\rho,
\end{align*}
and
\begin{align*}
D(m) & =18m^{6}+510m^{5}\text{\textminus}m^{5}\rho+100m^{4}\text{\textminus}135m^{4}\rho\text{\textminus}1050m^{3}\rho\\
 & \text{\textminus}500m^{3}+1250m^{2}\rho+6250m^{2}\text{\textminus}6250m\text{\textminus}3125m\rho+3125\rho,
\end{align*}
which prove \eqref{eq:2-2} and \eqref{eq:2-3}.

We finally prove \eqref{eq:2-5}. We subsitute \eqref{eq:t2-0} and
\eqref{eq:t2-1} into both sides of \eqref{eq:2-5} and then simplify
using the identity $\rho=(m^{3}-2m^{2}+5m)^{1/2}$ to derive that
both sides of \eqref{eq:2-5} are equal to
\[
\frac{(m-5)^{2}(m^{3}\text{\textminus}m^{2}+7m+2m\rho+1+2\rho)}{(m-1)^{4}},
\]
from which \eqref{eq:2-5} follows readily. This concludes the proof
of Theorem \ref{t2}.
\end{proof}

\subsection{Proof of Theorem \ref{t1}}

In this subsection we will prove Theorem \ref{t1}.

\noindent{\it Proof of Theorem \ref{t1}.} Using the equation \eqref{eq:t0-1}
we see that the identities \eqref{eq:1-7}\textendash \eqref{eq:1-5}
are respectively equivalent to
\begin{align}
\frac{\psi^{2}(q^{2})}{\psi^{2}(q^{10})}\frac{\psi^{8}(q^{5})}{\psi^{8}(q^{10})}\bigg(16q^{5}-\frac{\psi^{8}(q^{5})}{\psi^{8}(q^{10})}\bigg) & =\bigg(5q^{2}-\frac{\psi^{2}(q^{2})}{\psi^{2}(q^{10})}\bigg)\bigg(\frac{\psi^{2}(q^{2})}{\psi^{2}(q^{10})}-q^{2}\bigg)^{5},\label{eq:3-1}\\
\frac{\psi^{2}(q^{10})}{\psi^{2}(q^{2})}\frac{\psi^{8}(q)}{\psi^{8}(q^{2})}\bigg(16q-\frac{\psi^{8}(q)}{\psi^{8}(q^{2})}\bigg) & =\bigg(5q^{2}\frac{\psi^{2}(q^{10})}{\psi^{2}(q^{2})}-1\bigg)^{5}\bigg(1-q^{2}\frac{\psi^{2}(q^{10})}{\psi^{2}(q^{2})}\bigg),\label{eq:3-2}\\
\frac{\psi^{2}(q^{5})}{\psi^{2}(q)}\bigg(16q\frac{\psi^{8}(q^{2})}{\psi^{8}(q)}-1\bigg)^{2} & =\frac{\psi^{8}(q^{2})}{\psi^{8}(q)}\bigg(5q\frac{\psi^{2}(q^{5})}{\psi^{2}(q)}-1\bigg)^{5}\bigg(q\frac{\psi^{2}(q^{5})}{\psi^{2}(q)}-1\bigg),\label{eq:3-3}\\
\frac{\psi^{2}(q)}{\psi^{2}(q^{5})}\bigg(16q^{5}\frac{\psi^{8}(q^{10})}{\psi^{8}(q^{5})}-1\bigg)^{2} & =\frac{\psi^{8}(q^{10})}{\psi^{8}(q^{5})}\bigg(5q-\frac{\psi^{2}(q)}{\psi^{2}(q^{5})}\bigg)\bigg(q-\frac{\psi^{2}(q)}{\psi^{2}(q^{5})}\bigg)^{5},\label{eq:3-4}\\
\bigg(q\frac{\psi^{2}(q)}{\psi^{2}(q^{5})}-\frac{\psi^{2}(q^{2})}{\psi^{2}(q^{10})}\bigg)^{2} & =\frac{\psi^{2}(q^{2})}{\psi^{2}(q^{10})}\bigg(q-\frac{\psi^{2}(q)}{\psi^{2}(q^{5})}\bigg)\bigg(5q-\frac{\psi^{2}(q)}{\psi^{2}(q^{5})}\bigg).\label{eq:3-5}
\end{align}

We temporarily assume that $0<q<1.$ Let $\beta$ have 5 degree over
$\alpha.$ According to \cite[Theorem 5.4.2 (i) and (iii)]{B} we
have
\begin{align}
\psi(q) & =\sqrt{\frac{z_{1}}{2}}(\alpha/q)^{1/8},\label{eq:3-6}\\
\psi(q^{2}) & =\frac{1}{2}\sqrt{z_{1}}(\alpha/q)^{1/4},\label{eq:3-7}\\
\psi(q^{5}) & =\sqrt{\frac{z_{5}}{2}}(\beta/q^{5})^{1/8},\label{eq:3-8}\\
\psi(q^{10}) & =\frac{1}{2}\sqrt{z_{5}}(\beta/q^{5})^{1/4}.\label{eq:3-9}
\end{align}
It follows from \eqref{eq:3-7}, \eqref{eq:3-8} and \eqref{eq:3-9}
that
\begin{align}
\frac{\psi(q^{2})}{\psi(q^{10})} & =\sqrt{\frac{z_{1}}{z_{5}}}\bigg(\frac{\alpha}{\beta}\bigg)^{1/4}q,\label{eq:3-10}\\
\frac{\psi(q^{5})}{\psi(q^{10})} & =\frac{\sqrt{2}}{(\beta/q^{5})^{1/8}}.\label{eq:3-11}
\end{align}
Multiplying both sides of \eqref{eq:2-1} by $q^{12}$ and then using
\eqref{eq:3-10} and \eqref{eq:3-11} in the resulting equation we
can easily obtain the identity \eqref{eq:3-1}.

It is easily deduced from \eqref{eq:3-6} and \eqref{eq:3-7} that
\begin{equation}
\frac{\psi(q)}{\psi(q^{2})}=\frac{\sqrt{2}}{(\alpha/q)^{1/8}}.\label{eq:3-12}
\end{equation}
Then \eqref{eq:3-2} follows by substituting \eqref{eq:3-10} and
\eqref{eq:3-12} into \eqref{eq:2-2}.

It is easily seen from \eqref{eq:3-6} and \eqref{eq:3-8} that
\begin{equation}
\frac{\psi(q^{5})}{\psi(q)}=\sqrt{\frac{z_{5}}{z_{1}}}\bigg(\frac{\beta}{\alpha}\bigg)^{1/8}/q^{1/2}.\label{eq:3-13}
\end{equation}
Then \eqref{eq:3-3} follows easily by dividing both sides of \eqref{eq:2-3}
by $q$ and then using \eqref{eq:3-12} and \eqref{eq:3-13} in the
resulting identity.

Multiplying both sides of \eqref{eq:2-4} by $q$ and then employing
\eqref{eq:3-11} and \eqref{eq:3-13} in the resulting equation we
can attain \eqref{eq:3-4}.

The identity \eqref{eq:3-5} follows readily by multiplying both sides
of \eqref{eq:2-5} by $q^{4}$ and then using \eqref{eq:3-10} and
\eqref{eq:3-13} in the resulting identity.

From these we see that \eqref{eq:3-1}\textendash \eqref{eq:3-5}
holds for $0<q<1.$ By analytic continuation, these identities are
also true for $|q|<1.$ This completes the proof of Theorem \ref{t1}.
\qed

\section*{Acknowledgement}

 This work was partially supported by the National Natural Science
Foundation of China (Grant No. 11801451) and the Natural Science Foundation
of Hunan Province (Grant No. 2020JJ5682).

\end{document}